\documentclass[12pt]{article}
\usepackage[utf8]{inputenc}
\usepackage[a4paper]{geometry}
\usepackage{amsmath}
\usepackage{amsthm}
\usepackage{color}
\usepackage{amssymb}
\usepackage{graphicx}
\usepackage{centernot}
\usepackage{mathtools}

\newtheorem{theorem}{Theorem}[section]
\newtheorem{corollary}[theorem]{Corollary}

\newtheorem{lemma}[theorem]{Lemma}

\newtheorem{definition}[theorem]{Definition}
\newtheorem{remark}[theorem]{Remark}
\newtheorem{example}[theorem]{Example}
\newcommand{\p}{\mathcal{P}}
\newcommand{\h}{\hookrightarrow}
\newcommand{\hh}{\overset{\operatorname{\h}}}
\title{Primal-Proximity Spaces}
\author{Ahmad Al-Omari$^1$, Murad \"{O}zko\c{c}$^2$, Santanu Acharjee$^3$\\
$^1$Al al-Bayt University\\
Faculty of Sciences, Department of Mathematics\\
P.O. Box $130095$, Mafraq $25113$, Jordan\\
$^2$Mu\u{g}la S{\i}tk{\i} Ko\c{c}man University, Faculty of Science\\
Department of Mathematics  $48000$,
Mente\c{s}e-Mu\u{g}la, Turkey\\
$^3$Department of Mathematics, Guwahati-$781014$, Assam, India\\
e-mail: $^1$omarimutah1@yahoo.com, $^2$murad.ozkoc@mu.edu.tr\\
$^3$sacharjee326@gmail.com}
 
\date{}
\begin{document}

\maketitle
% -------------------------------------------------------------
% ----------------------------------------------------------------
\begin{abstract}
The main purpose of this paper is to introduce and study the primal-proximity spaces. Also, we define two new operators via primal proximity spaces and investigate some of their fundamental properties. In addition, we obtain a new topology, which is weaker than old one, via these new operators. Moreover, we not only discuss some of their properties but also enrich with some examples. 
\\

{\bf 2010 AMS Classifications:} 54E05, 54A05, 54A10, 54B99.\\

\end{abstract}
\maketitle

\section{Introduction}
In many areas of mathematics, topology plays  very crucial roles. Applications of many topological ideas, to solve  various problems of nature, have attracted  researchers of different branches of science and social sciences. Many new notions have been introduced in topology, which have enriched topology with several new areas of research. Some of the most important classical structures of topology are filters \cite{Wl}, ideals \cite{Ku}, and grills \cite{ch}. The definition of ideal was first introduced by Kuratowski \cite{Ku}. On the hand, the notion of grill was introduced in \cite{ch}. It is important to observe that the notion of ideal is the dual of filter, but ideal has helped researchers to introduce many new areas of topology viz. ideal topological space \cite{Ja},  $I$-proximity \cite{ktez}, etc. But to the best of our knowledge,  no literature was available on dual structure of grill prior to \cite {aoi}.\\

Recently, Acharjee et al. \cite{aoi}  introduced a new structure called `primal'. They obtain not only some fundamental properties related to primal but also some relationships between topological spaces and primal topological spaces. Primals \cite{aoi} come across as the dual of the notion of grills while the dual of filters are ideals. Later, Al-Omari et al. \cite{aao}  introduced several new operators in primal topological spaces using primal. On the other hand, the notion of proximity \cite{efr} is also an important notion in the area of topology as well. Several forms of this notion such as $I$-proximity \cite{ktez}, $\mu$-proximity \cite{mmd, yi}, quasi proximity \cite{ste}, and multiset proximity \cite{ktez2} have been studied by several researchers. Moreover, applications of proximity can be found in pattern recognition \cite{Pe}, region based theory of space \cite{ Di, Di1}, artificial intelligence \cite{ Du}, spatial analysis \cite{Bre}, etc.  One may refer to \cite{hosny1, hosny2, wj, b1, ktez1, ktez3, keyh, leader, lodato, m1, m2,  tiwari,   Na,  az, ta1, ta2, mn} and many others for proximity.\\

In section 3 of this paper, we introduce a new type of proximity  called primal-proximity. Also, we define point-primal proximity operator and investigate some of its fundamental properties in section 4. In addition, we prove that this operator is a Kuratowski closure operator under special condition. Moreover, we define one more operator via point-primal proximity operator. This operator come across as a Kuratowski closure operator without any condition. Furthermore, we give not only some relationships but also several examples. 

\section{Preliminaries}
In this section, we discuss some preliminary definitions which will be used in next sections.
\begin{definition} \cite{aoi}
Let $X$ be a non-emptyset.  A collection $\mathcal{P}\subseteq   2^X$  is called a primal on $X$ if it satisfies the following conditions:\\

(i) $X \notin \mathcal{P}$,\\

(ii) if $A\in \mathcal{P}$ and $B\subseteq A$, then $B\in \mathcal{P}$,\\

(iii) if $A\cap B\in \mathcal{P}$, then $A\in \mathcal{P}$ or $B\in \mathcal{P}$.
\end{definition}

\begin{corollary} \cite{aoi}
Let $X$ be a non-emptyset. A collection $\mathcal{P}\subseteq   2^X$  is a primal on $X$ if and only if it satisfies the following conditions:\\

(i) $X \notin \mathcal{P}$,\\

(ii) if $B\notin \mathcal{P}$ and $B\subseteq A$, then $A\notin \mathcal{P}$,\\

(iii) if $A\notin \mathcal{P}$ and $B\notin \mathcal{P},$ then $A\cap B\notin \mathcal{P}.$
\end{corollary} 

\begin{example}
\cite{aoi} Let $X$ be a non-emptyset. Then, $\p=\{A\subseteq X: |A^{c}|\geq\aleph_0\}$ is a primal on $X,$ where $\aleph_0$ is the lowest infinite cardinal number.
\end{example}

\begin{example}
\cite{aoi} Let $X$ be a non-emptyset. Then, $\p=\{A\subseteq X: |A^{c}|>\aleph_0\}$ is a primal on $X,$ where $\aleph_0$ is the lowest infinite cardinal number.
\end{example}

%Let $\delta$ be a binary relation on a non-emptyset $X$.
\begin{definition}
    \cite{efr} A binary relation $\delta$ on $2^{X}$ is called an (Efremovi$\check{c}$) proximity on $X$ if $\delta$ satisfies the following conditions:
\begin{enumerate}
  \item $A\delta B \Rightarrow B\delta A$,
  \item $A\delta (B\cup C) \Leftrightarrow A\delta B$ or $A\delta C$,
  \item $A\delta B \Rightarrow A\neq \emptyset$ and $B\neq\emptyset$,
   \item $A\cap B\neq \emptyset \Rightarrow A\delta B$,
    \item if $A\centernot\delta B$,  then there exists $C, D\subseteq X $ such that $A\centernot\delta C^{c}$, $D^{c}\centernot\delta B$ and $C\cap D=\emptyset$.
\end{enumerate}

\end{definition}
A proximity space is a pair $(X, \delta)$ consisting of a set $X$ and a proximity relation on $X$. We shall
write $A\delta B$ if the sets $A, B \subseteq X$ are $\delta$-related, otherwise we shall write $A\centernot\delta B$. Throughout this paper, the space $(X, \delta, \p)$ means an $Ef$-proximity space $(X, \delta)$ with a primal $\p$ on $X$. Now, we define the following definition which will be used in section 5.

\begin{definition}
 In a space $(X, \delta, \p)$, we say that a subset $A$ of $X$ is
locally in $\p$ at $x \in X$ if there exists a $\delta$-neighborhood $U$ of $x$ such that
$U^{c} \cup A^{c}\notin \p$. Also for a subset $A$ of $X$, the primal  local  function of $A$ with respect
to $\delta$ and $\p$, denoted by $A^{\diamond}(\delta, \p)$, simply $A^{\diamond}(\p)$
 or $A^{\diamond}$, is the set $\bigcup\{x \in X : A$
is not primal locally  in $\p$ at $x\}$ i.e., $A^{\diamond}
(\delta, \p) = \bigcup\{x \in X : U^{c} \cup A^{c} \notin \p$, for every
$\delta$-neighborhood $U$ of $x\}$.
\end{definition}

%\begin{example}
%Let $X$ be a non-emptyset. Then, $\p=\{A\subseteq X: A^{c}$ is infinite set$ \}$ is a primal on $X$.
%\end{example}

\section{Primal-Proximity Spaces}
In this section, we introduce the notion of primal-proximity on $X$ and investigate some of its fundamental properties.
\begin{definition}~\label{d1}
A binary relation $\hookrightarrow$ on $2^{X}$ with a primal $\mathcal{P}$ on a non-emptyset $X$ is called a  primal-proximity on $X$ if $\hookrightarrow$ satisfies the following conditions:
\begin{enumerate}
  \item[(1)] if $A\hookrightarrow B,$ then $B\hookrightarrow A$;
    \item[(2)] $A\hookrightarrow (B\cup C)$ if and only if
  $ A\hookrightarrow B$ or $A\hookrightarrow C$;
    \item[(3)] if $A^{c}\notin \p,$ then $A\centernot \h B $ for all  $B\subseteq X$;
    \item[(4)] if $(A\cap B)^{c}\in \p,$ then $ A\hookrightarrow B$;
      \item[(5)] if $A\centernot\hookrightarrow B,$ then there exist $  C, D\subseteq X $ such that $A\centernot\hookrightarrow C^{c}$, $D^{c}\centernot\hookrightarrow B$ and $(C\cap D)^{c}\notin \p$.
\end{enumerate}
\end{definition}

\begin{definition}
A primal-proximity space is a pair $(X,\hookrightarrow)$ consisting of a set $X$ and primal-proximity relation on a non-emptyset $X.$ We  write $A\hookrightarrow B$ if the sets $A,B\subseteq X$ are $\hookrightarrow$-related, otherwise we  write $A\centernot\hookrightarrow B.$
\end{definition}

\begin{remark}
Let $X$ be a non-emptyset and $A\subseteq X$ such that $\p=2^{X}\setminus \{X\}.$
\begin{enumerate}
\item  If $x\in A$, then $\{x\}\h A.$
\item If $A\centernot \h B $, then $A\cap B=\emptyset$.
\end{enumerate}
Suppose $A\cap B\neq\emptyset$. Then, there exists at least one point in $X$ such that $a \in A \cap B.$ Therefore, $(A\cap B)^c \neq X$. Hence, $(A\cap B)^{c} \in \p$ since $\p=2^X\setminus \{X\}.$ It follows that $A \h B$, which is impossible. Therefore, $A\cap B=\emptyset$.
\end{remark}

\begin{corollary}~\label{c1}
Let $\hookrightarrow$ be a primal-proximity on a non-emptyset $X.$ Then, the followings hold:
\begin{enumerate}
   \item[(1)] if $B\centernot \hookrightarrow A,$ then $A\centernot\hookrightarrow B,$

 \item[(2)] $ A\centernot\hookrightarrow (B\cup C)$ if and only if
  $ A\centernot\hookrightarrow B$ and $A\centernot\hookrightarrow C$,

 \item[(3)] if there exists $B\subseteq X$ such that $A \hookrightarrow B,$ then $A^{c}\in \p,$

 \item[(4)] if $ A\centernot \hookrightarrow B,$ then $(A\cap B)^{c}\notin \p ,$

 \item[(5)] if $A\centernot\hookrightarrow B,$ then there exist $  C, D\subseteq X $ such that $A\centernot\hookrightarrow C^{c}$, $D^{c}\centernot\hookrightarrow B$ and $(C\cap D)^{c}\notin \p$.
\end{enumerate}
\end{corollary}

\begin{example}
Let $\p$ be a primal on a non-emptyset $X$ and $A, B\subseteq X$. We define    a binary relation $\hookrightarrow$ on $2^{X}$
 as:
$$A\hookrightarrow B \Leftrightarrow  A^{c}, B^{c} \in\p.$$

Then, $\hookrightarrow$
is a primal-proximity relation. Indeed, one easily finds that $\hookrightarrow$ satisfies conditions, (1) to (4). We are to check that $\hookrightarrow$ also satisfies condition (5). Let $A \centernot\hookrightarrow B$. It follows that $A^{c} \notin \p$ or
$B^{c} \notin \p$. If $A^{c} \notin \p$, by taking $C = A^{c}$ and $D = A$
 have the required properties. If $B^{c} \notin \p$, by taking
$C = B$ and $D = B^{c}$.
\end{example}

\begin{example}~\label{exm1}
Let $\p$ be a primal on a non-emptyset $X$ and $A, B \subseteq X.$ We define a binary relation $\h$ on $2^{X}$
as:
$$A\hookrightarrow B \Leftrightarrow  (A\cap B)^{c} \in\p.$$

Then,  $\h$
is a primal-proximity on $X$. It follows directly from the definition that $\hookrightarrow$
satisfies conditions (1) to (4). To prove that $\hookrightarrow$ satisfies condition (5), let $A \centernot\hookrightarrow B$. It
follows that $(A \cap B)^{c} \notin \p$. If we take $C := B^{c}$
 and $D := B$,  the results proves.
\end{example}
\begin{example}
Let $(X,\tau,\mathcal{P})$ be a primal topological space such that $\mathcal{P}=2^X\setminus \{X\}.$ Let   $(X,\tau)$ be a normal space and $A, B\subseteq X$.  Define a binary relation
 $\h$ on $2^{X}$ as:
$$A\h B \Leftrightarrow  \left(cl(A)\cap cl(B)\right)^{c} \in\p,$$ where the  closure is taken with respect to $\tau$.
Then, the binary relation $\h$ is a primal-proximity on $X.$
\end{example}

\begin{proof}
\begin{enumerate}
  \item[(1)] $A\h B \Leftrightarrow  \left(cl(A)\cap cl(B)\right)^{c} \in\p \Leftrightarrow \left(cl(B)\cap cl(A)\right)^{c} \in\p \Leftrightarrow B\h A$.
   \item[(2)] Let $A,B,C\subseteq X.$ \\
  $\begin{array}{rcl} A\hookrightarrow (B\cup C) & \Leftrightarrow & \left(cl(A)\cap cl(B\cup C)\right)^{c}\in\p \\ & \Leftrightarrow & \left(cl(A)\cap \left(cl(B)\cup cl(C)\right)\right)^{c} \in\p \\ & \Leftrightarrow & \left(\left(cl(A)\cap cl(B)\right)\cup \left(cl(A)\cap cl(C)\right)\right)^{c} \in\p \\ & \Leftrightarrow & \left(cl(A)\cap cl(B)\right)^{c}\cap \left(cl(A)\cap cl(C)\right)^{c} \in\p \\ & \Leftrightarrow &  \left(cl(A)\cap cl(B)\right)^{c}\in \p \text{ or }  \left(cl(A)\cap cl(C)\right)^{c} \in\p \\ & \Leftrightarrow &  A\hookrightarrow B \text{ or } A\hookrightarrow C.\end{array}$
  
  \item[(3)] Let $A\h B.$ \\
$\left.\begin{array}{rr} A\h B \Leftrightarrow \left(cl(A)\cap cl(B)\right)^{c} \in\p  \Rightarrow \left(cl(A)\right)^{c} \in\p \\ \p =2^X\setminus \{X\}\end{array}\right\} \Rightarrow \left(cl(A)\right)^{c}\neq X$
  \\
$\left.\begin{array}{rr}\Rightarrow cl(A)\neq\emptyset \Rightarrow A\neq\emptyset \Rightarrow  A^{c}\neq X \\ \p =2^X\setminus \{X\}\end{array}\right\} \Rightarrow A^c \in\p.$
%There is a problem in this proof. If $(\overline{A})^{c} \in\p$, then we can not obtain $A^{c} \in\p$ since $\p$ is downward closed.
  
  \item[(4)] Let $A\centernot \h B.$ \\
%$A\centernot \h B \Rightarrow  \left(cl(A) \cap cl(B)\right)^{c} \notin\p  \Rightarrow \left(\overline{A} \cap \overline{B}\right)^{c}=X \Rightarrow \overline{A} \cap \overline{B}=\emptyset   \Rightarrow  \left(\overline{A}\cap \overline{B}\right)^{c} \notin\p$.
$\left.\begin{array}{rr} A\centernot \h B \Rightarrow  \left(cl(A) \cap cl(B)\right)^{c} \notin\p  \\ \p =2^X\setminus \{X\}\end{array}\right\}\Rightarrow \left(cl(A) \cap cl(B)\right)^{c}=X  $
\\
$\left.\begin{array}{rr}\Rightarrow cl(A) \cap cl(B)=\emptyset\Rightarrow A \cap B=\emptyset\Rightarrow (A \cap B)^c=X \\ \p =2^X\setminus \{X\}\end{array}\right\} \Rightarrow (A\cap B)^c \notin\p.$
 \\  
  \item[(5)] Let $A\centernot \h B$. Then, $\left(cl(A)\cap cl(B)\right)^{c} \notin\p$. So $cl(A)\cap cl(B)=\emptyset$.
    So that, since $(X,\tau)$ is normal space there exist two disjoint open sets in $\tau$, $C$ and $D$ such that $cl(A)\subseteq C$ and $cl(B)\subseteq D$.
     Hence, $C^{c}$ is closed and $cl(A)\cap C^{c}=\emptyset$. This implies $cl(A)\centernot\h C^{c}$. Since $C\cap D=\emptyset$, we have $C\subseteq D^{c}$.
     It follows that $cl(C)\subseteq D^{c}$ since $D^{c}$ is closed. Therefore, $cl(C) \cap cl(B) = \emptyset$ and $ \left(cl(C)\cap cl(B)\right)^{c} \notin\p.$ Hence, $C \centernot\h B$.
     Let $E = C^{c}$. Then, $A \centernot\h B$ implies that there exists a subset $E$ such that $A \centernot\h E$ and $E^{c} \centernot\h B$ and $\left(E\cap E^{c}\right)^{c}\notin \p$.\qedhere
\end{enumerate}
\end{proof}

\section{Point-Primal Proximity Operator}

This section introduces point-primal proximity operator. Here, we study several properties of a primal-proximity space using this operator. \\

\begin{definition}
Let $(X,\hookrightarrow)$ be a primal-proximity space. Then, the operator  $\hh{(\cdot)}:2^X\to 2^X$ defined by $\hh{A}:=\{x\in X|\{x\}\hookrightarrow A\}$ is said to be point-primal proximity operator. Moreover,  $\hh{A}$ is said to be point-primal proximity of $A$.
\end{definition}

We now provide the following lemma without the proof.

\begin{lemma}\label{1}
Let $\mathcal{P}$ be a primal on a non-emptyset $X.$ If $A\h B,$ $A\subseteq C,$ and $B\subseteq D,$ then $C\h D.$
\end{lemma}

\begin{lemma}\label{2}
Let $(X,\hookrightarrow, \mathcal{P})$ be a primal-proximity space and $A, B\subseteq X$. If $B\centernot \h A,$ then $\hh{A}\subseteq B^{c}$.
\end{lemma}
\begin{proof}
Suppose  $\hh{A}\cap B\neq \emptyset$. Then there exists at least a point $x\in \hh{A}\cap B$. So $x\in\hh{ A}$ and   $x\in B$, i.e.
$\{x\} \h A$ and $\{x\}\subseteq B$ by Lemma ~\ref{1} implies that $A\h B$,  which is a contradiction. Hence,
$\hh{A}\subseteq B^{c}$.
\end{proof}
%%%%%%%%%%%%%%%%%%%%%%%%%%%%%%%%%%%%%%%%%%%%%%%%%%%%%%%%%%%%%%%%%%%%%%%%%%%%%%%%%%%%%%%%%
\begin{theorem}\label{th1}
Let $(X,\hookrightarrow, \mathcal{P})$ be a primal-proximity space and $A, B\subseteq X$. If
$B\centernot \h A,$ then $B \centernot \h \hh{A}$.
\end{theorem}
\begin{proof}
Let $B\centernot \h A.$ Then by (5) of Definition  ~\ref{d1}  there exist
$C, D\subseteq X$ such that $B \centernot \h C^{c}$, $D^{c} \centernot \h A$ and $(C\cap D)^{c}\notin \p$.
This result, combined with Lemma ~\ref{2}, implies that $\hh{A}\subseteq D$. Now, we want
to prove that $\hh{A}\subseteq C^{c}$. Let $x\in \hh{A}$, then $\{x\}\h A$. Suppose $x\in C$ then implies that
$x\in C\cap D$ and $ (C\cap D)^{c}\subseteq X\setminus\{x\}$, so $X\setminus\{x\}\notin \p$. Then,  by Definition  ~\ref{d1} (3), we have $\{x\}\centernot \h A$, which a contradiction. Hence, $x\in C^{c}$. So
$\hh{A}\subseteq C^{c}$. Now, we have by Lemma ~\ref{1} $B \centernot \h \hh{A}$. Hence, the theorem is proven.
\end{proof}
%%%%%%%%%%%%%%%%%%%%%%%%%%%%%%%%%%%%%%%%%%%%%%%%%%%%%%%
%%%%%%%%%%%%%%%%%%%%%%%%%%%%%%%%%%%%%%%%%%%%%%%%%%%%%%%%%%%%%%%%%%%%%%%%%%%%%%%%%%%%%%%%%
Due to Theorem ~\ref{th1} and (1) of Definition  ~\ref{d1}, we have the following corollary.
\begin{corollary}
Let $(X,\hookrightarrow, \mathcal{P})$ be a primal-proximity space and $A, B\subseteq X$. If
$B\centernot \h A,$ then $\hh{B} \centernot \h \hh{A}$.
\end{corollary}

\begin{theorem}~\label{th5}
Let $(X,\hookrightarrow, \mathcal{P})$ be a primal-proximity space and $A, B\subseteq X$. Then, the following properties hold:
\begin{enumerate}
           \item if $A\subseteq B,$ then $\hh{A}\subseteq \hh{B};$
           \item  $\hh{(A\cap B)} \subseteq  \hh{A}\cap \hh{B};$
           \item $\hh{A}\cup \hh{B}= \hh{(A\cup B)};$
           \item $\hh{\hh{A}} \subseteq  \hh{A};$    
           \item  if $A^c\notin\p,$ then $\hh{A}=\emptyset;$
           \item $\hh{\emptyset} = \emptyset;$
           \item $\hh{A}\setminus \hh{B}\subseteq \hh{(A\setminus B)};$
           \item if $B^{c}\notin \p$, then  $ \hh{(A\cup B)}=\hh{A}=\hh{(A\setminus B)};$
            \item if $ [(A\setminus B)\cup (B\setminus A)]^{c}\notin \p$, then $\hh{A}=\hh{B}.$
         \end{enumerate}
\end{theorem}

\begin{proof}
1) Let $A\subseteq B$ and $x\in \hh{A}.$ Then $\{x\}\h A.$ Since $A\subseteq B,$ by Lemma \ref{1}, we have $\{x\}\h B.$ Hence, $x\in \hh{B}.$ 
\\

%\\
%$\left.\begin{array}{r}
%x\in \hh{A}\Rightarrow \{x\}\h A \\
%A\subseteq B
%\end{array}\right\}\overset{\text{Lemma }\ref{1}}{\Rightarrow} %\{x\}\h B\Rightarrow x\in \hh{B}.$
%\\
2) Let $A,B\subseteq X.$ From $(1)$, it is not difficult to see that $\hh{(A\cap B)}\subseteq \hh{A}$  and $\hh{(A\cap B)}\subseteq \hh{B}$. Thus, we get $\hh{(A\cap B)}\subseteq\hh{ A}\cap \hh{B}.$
\\
%$
%\left.\begin{array}{r}
%A,B\subseteq X\Rightarrow A\cap B\subseteq A\overset{(1)}{\Rightarrow} \hh{(A\cap B)}\subseteq \hh{A}
%\\
%A,B\subseteq X\Rightarrow A\cap B\subseteq B\overset{(1)}{\Rightarrow} \hh{(A\cap B)}\subseteq \hh{B}
%\end{array}\right\}\Rightarrow \ldots (*)
%$
%\\
%
%Now, let $x\in \hh{A}\cap \hh{B}.$
%\\
%$x\in \hh{A}\cap \hh{B}\Rightarrow (x\in \hh{A})(x\in \hh{B})\Rightarrow (\{x\}\h A)(\{x\}\h B)\Rightarrow x\in \hh{(A\cap B)}$
%\\
%Then we have $\hh{A}\cap \hh{B}\subseteq \hh{(A\cap B)}\ldots (**)$
%\\
%$(*),(**)\Rightarrow \hh{(A\cap B)}= \hh{A}\cap \hh{B}.$
%\\

3) Let $A,B\subseteq X.$ From $(1)$, we can easily find that $\hh{A}\subseteq \hh{(A\cup B)}$  and $\hh{B}\subseteq \hh{(A\cup B)}$. Thus, obviously $\hh{A}\cup \hh{B} \subseteq \hh{(A\cup B)}$.

Conversely, let $y\in \hh{(A\cup B)}$. Then, $\{y\}\hookrightarrow A\cup B$. Due to Definition 3.1, either $\{y\}\hookrightarrow A$ or $\{y\}\hookrightarrow  B$. It indicates that either $y\in \hh{A}$ or $y\in \hh{B}$. So, we can conclude that $ \hh{(A\cup B)}\subseteq\hh{ A}\cup \hh{B}$.
\\
%Let $A,B\subseteq X.$ \\
%$
%\left.\begin{array}{r} A,B\subseteq X\Rightarrow A\subseteq A\cup B\overset{(1)}{\Rightarrow} A^{\h}\subseteq (A\cup B)^{\h}
%\\
%A,B\subseteq X\Rightarrow B\subseteq A\cup B\overset{(1)}{\Rightarrow} B^{\h}\subseteq (A\cup B)^{\h} \end{array}\right\}\Rightarrow %A^{\h}\cup B^{\h}\subseteq (A\cup %B)^{\h}.
%$
%\\
%Now let $x\in (A\cup B)^{\h}.$
%\\
%$x\in (A\cup B)^{\h}\Rightarrow x\h A\cup B$

4) Let $A\subseteq X$ and let $x\notin \hh{A}$. Then, $\{x\}\centernot \h A$. So, by Theorem ~\ref{th1}
we have  $\{x\}\centernot \h \hh{A}$. Hence, $x\notin \hh{\hh{A}}$. Thus, we get
$\hh{\hh{A}}\subseteq \hh{A}$.
\\

5) Let $A^c\notin\p.$ Then, by (3) of Definition \ref{d1}, $A\centernot \h B$ for all subsets $B$ of $X.$ Therefore, we have $A\centernot \h \{x\}$ for all $x\in X.$ Again, by (1) of Definition \ref{d1}, $\{x\}\centernot \h A$ for all $x\in X.$ This means $x\notin \hh{A}$ for all $x\in X.$ Hence, $\hh{A}=\emptyset.$
\\

%%%Please do not delete%%%
%$A^c\notin\p\Rightarrow (\forall B\subseteq X)(A\centernot \h B)\Rightarrow (\forall x\in X)(A\centernot \h \{x\})\Rightarrow (\forall x\in X)(\{x\}\centernot \h A)\Rightarrow \hh{A}=\emptyset.$\\

6) Let $\hh{\emptyset} \neq \emptyset.$ Thus, we assume that $x\in \hh{\emptyset}$. Hence, $\{x\}\hookrightarrow \emptyset$. By (1) of  Definition 3.1, $\emptyset \hookrightarrow \{x\}$. Again, by (3) of Definition 3.1, $\emptyset^{c}= X\in \mathcal{P}$; which is a contradiction to the definition of primal $\mathcal{P}$. Hence, $\hh{\emptyset}= \emptyset.$

7) For all $A, B \subseteq X$, $A=(A\setminus B)\cup (A\cap B)$, by (3) we have $\hh{A}=\hh{(A\setminus B})\cup \hh{(A\cap B)}\subseteq \hh{(A\setminus B})\cup \hh{B}$. Hence, $\hh{A}\setminus\hh{B}\subseteq \hh{(A\setminus B})$.

8) If  $B^{c}\notin \p$, then  $ \hh{(A\cup B)}=\hh{A}\cup \hh{B}=\hh{A}\cup \emptyset=\hh{A}$. Also, $\hh{A}\setminus\hh{B}\subseteq \hh{(A\setminus B})$, then $\hh{A}\subseteq \hh{(A\setminus B)}$. And  $\hh{(A\setminus B})=\hh{(A\cap B^{c})}\subseteq \hh{A}\cap \hh{B^{c}}\subseteq \hh{A}$.
Hence,  $ \hh{(A\cup B)}=\hh{A}=\hh{(A\setminus B)}.$

9) If $ [(A\setminus B)\cup (B\setminus A)]^{c}\notin \p$, then  $ (A\setminus B)^{c}\notin \p$ and  $(B\setminus A)^{c}\notin \p$. Since $\hh{A}=\hh{[(A\setminus B)\cup (A\cap B)]}$
and $ (A\setminus B)^{c}\notin \p$, by using (8) $\hh{A}=\hh{ (A\cap B)}\subseteq \hh{B}.$  It follows that $\hh{A}\subseteq\hh{B}$.
Similarly, since $\hh{B}=\hh{[(B\setminus A)\cup (B\cap A)]}$
and $ (B\setminus A)^{c}\notin \p$, by using (8) $\hh{B}=\hh{ (B\cap A)}\subseteq \hh{A}.$  It follows that $\hh{B}\subseteq\hh{A}$.
Hence, $\hh{A}=\hh{B}$.
\end{proof}

%%%%%%%%%%%%%%%%%%%%%%%%%%%%%%%%%%%%%%%%%%%%%%%%%%%%%%%%%%%%%%%%%%%%%%%%%%%%%%%%%%%%%%%%%

%%%%%%%%%%%%%%%%%%%%%%%%%%%%%%%%%%%%%

%%%%%%%%%%%%%%%%%%%%%%%%%%%%%%%%%%%%%%%%%%%%%%
%%%%%%%%%%%%%%%%%%%%%%%%%%%%%%%%%%%%%%%%%%%%%%%%%%%%%%%%%%%%%%%%%%%%%%%%%%%%%%%%%%%%%%%%%

%%%%%%%%%%%%%%%%%%%%%%%%%%%%%%%%

%%%%%%%%%%%%%%%%%%%%%%%%%%%%%%%%%%%%%%%%%%%%%%%%%%%%%%%%%%%%%%%%%%%%%%%%%%%%%%%%%%%%%%%%%

%%%%%%%%%%%%%%%%%%%%%%%%%%%%%%%%%%%%%%%%%%%%%%%%%%%%%%%%%%%%%%%%%%%%%%%%%%%%%%%%%%%%%%%%%

\begin{remark}
Let $(X,\hookrightarrow, \mathcal{P})$ be a primal-proximity space and $A\subseteq X$. The inclusion $A\subseteq \hh{A}$ need not be true in general as shown by following example.
\end{remark}

\begin{example}
Let $X=\{a, b, c\}$, $\p=\{\emptyset, \{b\}, \{c\}, \{b, c\}\}$ and the binary relation $\h$ on $2^X$  defined as Example \ref{exm1}. For the subset $A=\{b\},$ we have $A=\{b\}\nsubseteq \emptyset=\hh{A}.$  
\end{example}

\begin{theorem}
Let $(X,\hookrightarrow, \mathcal{P})$ be a primal-proximity space and $A, B\subseteq X$. Then, the following statements hold:\begin{enumerate}
           \item $A\cap \hh{B}=\emptyset$, for all $A^{c}\notin \p$ and $B\subseteq X$,
           \item $\{x\}\h X$ for all $x\in X$ if and only  if $\p= 2^X\setminus \{X\}$.
           \item if $\mathcal{P}=2^X\setminus \{X\},$ then $\hh{X}=X.$
         \end{enumerate}

\end{theorem}
\begin{proof}
%(1) Suppose $A^{c}\notin \p$ and $A\cap \hh{B}\neq\emptyset.$ It follows that $A\centernot \h B$  and also there exists $x\in X$ such that $x\in A$ and $\{x\} \h B$. Hence, by Lemma ~\ref{1} we have $A\h B$, which is a contradiction. So, $A\cap \hh{B}=\emptyset$. \color{red} I rearranged this proof as below. I will delete this proof after you check. \color{black}

(1) Let $A^{c}\notin \p$ and suppose $A\cap \hh{B}\neq\emptyset.$ It follows that $A\centernot \h B$ since $A^{c}\notin \p$ and also $\hh{B}\nsubseteq A^c$. Hence, by Lemma \ref{2} we have
$A\h B$, which is a contradiction. Thus, $A\cap \hh{B}=\emptyset$.
\\

(2) If $\{x\}\h X$ for all $x\in X$, then by (3) of  Corollary  \ref{c1}, we have
$\{x\}^{c}\in \p$ for all $x\in X$. Hence, $\p= 2^X\setminus \{X\}$. Conversely, if
$\p= 2^X\setminus \{X\}$, then $(\{x\}\cap X)^{c}=(\{x\})^{c}\in \p$ and   by  (4) of Definition  \ref{d1}, we have $\{x\}\h X$ for all $x\in X$.
\\

 (3) Let $x\in X.$ Since $\mathcal{P}=2^X\setminus\{X\},$ then $(\{x\})^{c}=(\{x\}\cap X)^{c}\in \p$ and by (4) of Definition \ref{d1}, we get $\{x\}\h X$ for all $x\in X$. Hence, $\hh{X}=X.$
\end{proof}
%%%%%%%%%%%%%%%
%%%%%%%%%%%%%%%%%%%%%%
\begin{theorem}
Let $(X,\hookrightarrow, \mathcal{P})$ be a primal-proximity space. If $A, B, C \subseteq X$ and  $B\subsetneqq C$  such that
$A\centernot \h B$ but $A \h C$, then $A \h (C\setminus B)$.
\end{theorem}

\begin{proof}
Let $A, B, C \subseteq X$ and  $B\subsetneqq C$.. We consider  $D=C\setminus B$. Since
$A \h C$, then  $A \h B\cup (C\setminus B)=B\cup D$.
Then by Definition  \ref{d1},  $A \h B$ or $A \h D$. Now, $A \h B$ is not possible since we consider $A \centernot\h B$. Then obviously, $A \h (C\setminus B)$.
\end{proof}

\begin{theorem}\label{t11}
Let $(X,\h, \mathcal{P})$ be a primal-proximity space and $A, B\subseteq X$. If $A \centernot\h  B$, then there exists $C\subseteq X$ such that $A \centernot\h  C$ and $B \centernot\h  C^{c}$.
\end{theorem}

\begin{proof}
Since $A \centernot\h B$, thus by (5) of Definition 3.1, there exist $M, N\subseteq X $ such that $A \centernot\h  M^{c}$, $N^{c} \centernot\h  B$  and $(M\cap N)^c\notin \mathcal{P}$. Let $M= X\setminus C $ and $N=C$. Then, $(M\cap N)^c = X\notin \mathcal{P}$. Also, $A \centernot\h  (X\setminus C)^{c}$, $C^{c} \centernot\h  B$. It yields $A \centernot\h   C$, $B\centernot\h  C^{c} $. Hence, the proof is completed.
\end{proof}

\begin{corollary}
Let $(X,\h, \mathcal{P})$ be a primal-proximity space and $A, B, C\subseteq X$. If $A \centernot\h  B$ and $B \h  C$, then  $A \centernot\h  C$.
\end{corollary}

%\begin{theorem}
%Let $(X,\h, \mathcal{P})$ be a primal-proximity space  \color{red} such that $\mathcal{P}=2^X\setminus \{X\}$ \color{black} and $A, B\subseteq X$. If $A\h  B$, then  $A \centernot\h  B^{c}$.  {\color{red} We will delete this result.} \end{theorem}

%%%%%%%%%%%%%%%%%%%%%
\section{Proximal Closed Sets and  $\hh \tau$ Topology}

In this section, proximal closed sets are defined. Moreover, various results between a primal-proximity space and a primal topological space are obtained using proximal closed sets and related notions.\\

%\begin{definition}
%Let $(X,\h, \mathcal{P})$ be a primal-proximity space. Then, $X$ is called separated if it satisfies  $\{x\} \h \{y\}$ implies $x = y$, and $(X,\h, \mathcal{P})$ is called a separated primal-proximity space. {\color{red}We will delete this Definition.}
%\end{definition}
%%%%%%%%%%%%%%%%%%%%%%%%%%%%%%%%%%%%%%%%%%%%%%%%%%%%%%%%%%%%%%%%%%%%%%%%%%
%%%%%%%%%%%%%%%%%%%%%%%%%%%%%%%%%%%%%%%%%%%%%%%%%%
\begin{definition}
Let $(X,\h, \mathcal{P})$ be a primal-proximity space. Then, a subset $F$ of $X$ is called proximity-closed if and only if  $\{x\} \h F$ implies $x \in F$.
\end{definition}
%%%%%%%%%%%%%%%%%%%%%%%%%%%%%%%%%%%%%%%%%%%%%%%%%%%%%%%%%%%%%%%%%%%%%%%%%%
%%%%%%%%%%%%%%%%%%%%%%%%%%%%%%%%%%%%%%%%%%%%%%%%%%
\begin{lemma} \label{lem2}
If there is a point $x\in X$ such that $A \h  \{x\}$ and $\{x\} \h B$, then $A \h B$.
\end{lemma}

\begin{proof}
Suppose $A\centernot \h B$, by Theorem \ref{t11}, there exists a subset $C$ such that
$A\centernot \h C$ and $C^c\centernot \h B$. Now,  either $x\in C$ or  $x\in C^c$.

Case (1): If $x\in C$, then $A \centernot\h  \{x\}$. For if $A \h  \{x\}$, then $A \h  C$, by Lemma ~\ref{1} which is a contradiction.

Case (2): If $x\in C^c$, then $\{x\} \centernot\h  B$. Therefore, if $A \h  \{x\}$ and $\{x\} \h B$, then $A \h B$.
\end{proof}
%%%%%%%%%%%%%%%%%%%%%%%%%%%%%%%%%%%%%%%%%%%%%%%%%%%%%%%%%%%%%%%%%%%%%%%%%%
%%%%%%%%%%%%%%%%%%%%%%%%%%%%%%%%%%%%%%%%%%%%%%%%%%
\begin{theorem} \label{th3}
The collection of complements of all proximity-closed sets of $(X,\h, \mathcal{P})$  forms a topology on $X$. This topology is denoted by $\hh{\tau}$.
\end{theorem}

\begin{proof}
Since $X$ and $\emptyset$ are proximity-closed in $(X,\h, \mathcal{P})$, their complements $\emptyset$ and $X$ are in $\hh{\tau}$.

Let $\{F_{i}: i\in I\}$ be a collection of proximity-closed sets. If $\{x\} \h \bigcap\{F_{i}: i\in I\},$ then
 $\{x\} \h F_{i}$ for every $i \in I$ by Lemma ~\ref{1}. Since $F_{i}$ is proximity-closed, $x\in F_{i}$ for every $i \in I$. Hence, $x \in \bigcap\{F_{i}: i\in I\}$ and $\bigcap\{F_{i}: i\in I\}$ is proximity-closed. Therefore, if $(X\setminus F_{i})\in \hh{\tau}$ for every $i \in I$, then $\bigcup \{X\setminus F_{i}: i\in I\}$ is the complement of
$\bigcap\{F_{i}: i\in I\}$ which  belongs to $\hh{\tau}$.

Finally, let $F_{1}$ and $F_{2}$ be two proximity-closed sets. If $\{x\} \h  F_{1}\cup F_{2},$ then $\{x\} \h  F_{1}$ or $\{x\} \h  F_{2}$. Thus,
$x \in  F_{1}$ or $x \in  F_{2}$ since $F_{1}$ and $F_{2}$ are proximity-closed. This implies $x \in F_{1}\cup F_{2}$. Thus, $F_{1}\cup F_{2}$ is proximity-closed.
Therefore, if $X\setminus F_{1} \in \hh{\tau}$ and $X\setminus F_{2} \in \hh{\tau},$  then 
 $(X\setminus F_{1}) \cap (X\setminus F_{2})=X\setminus (F_{1}\cup F_{2}) \in \hh{\tau}$. Hence, $\hh{\tau}$  is
a topology on $X$.
\end{proof}
%%%%%%%%%%%%
%%%%%%%%%%%%%%%%%%%%%%%%%%%%%%%%%%%%%%%%%%%%%%%%%%
\begin{theorem}
Let $(X,\h, \mathcal{P})$ be a primal-proximity space. The set $\hh{A}$ is the closure of $A$ where the closure is taken with respect to the topology  $\hh{\tau}$ and denoted by $cl_{\hh{\tau}}(A).$ 
\end{theorem}
\begin{proof}
Let $x\in  \hh{A}.$ Then $\{x\}\h A$. By Lemma \ref{1}, $\{x\}\h cl_{\hh{\tau}}(A)$
since $A\subseteq cl_{\hh{\tau}}(A)$ and $cl_{\hh{\tau}}(A)$ is proximity-closed in $\hh{\tau}.$ Thus, $x\in cl_{\hh{\tau}}(A).$ Hence, $\hh{A}\subseteq cl_{\hh{\tau}}(A)$.

Conversely, let $x\notin \hh{A}.$ Then $\{x\}\centernot \h A$. By Theorem  \ref{t11},  there exists a subset $C$ such that $\{x\}\centernot \h C$  and $C^c\centernot \h A$. Since there is no point of $C^c$ which is related to
$A$, then $\hh{A}\subseteq C$. By Lemma \ref{1},  $\{x\}\centernot \h cl_{\hh{\tau}}(A)$.
Thus, $\hh{A}$ is proximity-closed in $\hh{\tau}$. Therefore, $cl_{\hh{\tau}}(A)\subseteq \hh{A}$. Hence,  $cl_{\hh{\tau}}(A)=\hh{A}$.
\end{proof}
\begin{definition} \cite{Ku}
The the operator $\Phi:2^X\to 2^X$ is a Kuratowski closure operator provided:
\begin{enumerate}
  \item[(1)] $\Phi(\emptyset)=\emptyset$;
  \item[(2)] $A\subseteq \Phi(A)$  for every $A \in 2^X$;
  \item[(3)] $\Phi(A\cup B)=\Phi(A)\cup\Phi(B)$ for any $A,B \in 2^X$;
  \item[(4)] $\Phi(\Phi(A))=\Phi(A)$ for every $A \in 2^X$.
\end{enumerate}
\end{definition}
 \begin{theorem}
Let $(X,\h, \mathcal{P})$ be a primal-proximity space such that $\mathcal{P}=2^X\setminus \{X\}$. Then, the  operator
  $\hh{A}:=\{x\in X|\{x\}\hookrightarrow A\}$ on a primal-proximity space $(X, \h, \p)$
  is a Kuratowski closure operator.
\end{theorem}

\begin{proof}
    (1) By (6) of Theorem \ref{th5}, $\hh{\emptyset}=\emptyset$.

    (2) If $x \in A$, then $\{x\} \h A$. Hence, $x \in \hh{A}$. This shows that $A \subseteq \hh{A}$.

    (3) By (3) of Theorem \ref{th5},  $\hh{(A\cup B)}=\hh{A}\cup \hh{B}$.

    %\color{blue} or $\{x\}\in \hh{(A\cup B)}$ if and only if $\{x\}\h (A\cup B)$ if and only if $\{x\}\h A $ or $\{x\}\h  B$ if and only if $\{x\}\in  A $ or $\{x\}\in   B$ if and only if $\{x\}\in \hh{A}\cup \hh{B}$. Hence $\hh{(A\cup B)}=\hh{A}\cup \hh{B}$. \color{black} \color{red} No need for lines marked in blue.\color{black}

    (4) By (4) of Theorem \ref{th5}, we have always $\hh{\hh{A}} \subseteq \hh{A}.$
    Now, let $x\notin \hh{\hh{A}}.$ Then $\{x\}\centernot \h {\hh{A}}.$ By (4) of Corollary  
    \ref{c1}, we have $\left(\{x\}\cap {\hh{A}}\right)^c\notin\mathcal{P}.$ Since $\mathcal{P}=2^X\setminus\{X\},$ we get $\left(\{x\}\cap {\hh{A}}\right)^c=X$ which means that $\{x\}\cap {\hh{A}}=\emptyset.$ Thus, we have $x\notin {\hh{A}}.$ Hence, $\hh{A}\subseteq\hh{\hh{A}} $ and $\hh{\hh{A}}=\hh{A}$
 which completes the proof and this topology is denoted by $\hh{\tau}$. \color{black}
% Then $\{x\}\centernot {\hh{A}}.$  Therefore, $\hh{A}\subseteq \hh{\hh{A}}$. \color{blue} If $x\notin \hh{A}$, then $\{x\}\centernot \h A$. This implies that there exists a subset $E$ such that $\{x\}\centernot \h E$ and $E^{c}\centernot \h A$ and $(E\cap E^{c})^{c}\notin \p$. Now, if $A$ is not contained in $E$, then there exists an element $a\in \hh{A}$  but $a \notin  E.$ Hence, $\{a\}\h A$ and $a\in E^{c}$, contradicting $E^{c}\centernot \h A$. Hence, $\hh{A} \subseteq E$. By Lemma 2.2, $\{x\} \centernot \h \hh{A}$ since $\{x\} \centernot \h E$. \color{black} We should delete lines marked blue because we proved it in (4) of Theorem \ref{th5}.
%This means that $x\notin \hh{\hh{A}}.$  Hence, $\hh{\hh{A}}\subseteq \hh{A}$ and $\hh{\hh{A}}=\hh{A}$ which is completes the proof and this topology is denoted by $\hh{\tau}$.
\end{proof}
%%%%%%%%%%%%%%%%%%%%%%%%%%%%%%%%%%%%%%%%%%%%%%%%%%
\begin{theorem}
Let $(X,\h,\mathcal{P})$ be a primal-proximity space.
%%NO NEED THIS CONDITION \color{red} such that $\mathcal{P}=2^X\setminus \{X\}$. \color{black}  
Then, the operator $cl^{*}: 2^{X}\to 2^{X}$
defined by $cl^{*}(A):=A\cup \hh{A}$ satisfies Kuratowski closure 
axioms and induces a topology on $X$ called $\tau^{*}$ is given by $\tau^{*}=\{A\subseteq X| cl^{*}(A^{c})=A^{c}\}.$
\end{theorem}
\begin{proof}
    (1) By (6) of Theorem \ref{th5}, we have $cl^{*}(\emptyset)=\emptyset\cup \hh{\emptyset}=\emptyset$.

    (2) Let $A\subseteq X.$ Since $cl^*(A):=A\cup \hh{ A},$ we have $A\subseteq cl^*(A).$
    %By definition of operator $cl^{*}$, implies that $A\subseteq cl^{*}(A)$ for all $A\subseteq X.$

    (3)  Let $A,B\subseteq X.$ By (3) of Theorem \ref{th5}, we have
    $$\begin{array}{rcl} cl^{*}(A\cup B) & = & (A\cup B) \cup\hh{(A\cup B)} \\ & = & (A\cup B) \cup \left(\hh{A}\cup \hh{B}\right) \\ & = & \left(A\cup \hh{A}\right)\cup \left(B\cup \hh{B}\right)
    \\ & = & cl^{*}(A)\cup cl^{*}{(B)}.\end{array}$$
\color{black}

    (4) Let $A\subseteq X.$ By (4) of Theorem \ref{th5}, we have 
    $$\begin{array}{rcl} cl^*(cl^*(A)) & = & cl^*(A)\cup \hh{cl^*(A)} \\ & = & \left(A\cup \hh{A}\right)\cup \hh{\left(A\cup \hh{A}\right)} \\ & = & \left(A\cup \hh{A}\right)\cup \left(\hh{A} \cup \hh{\hh{A}}\right)
    \\ & = & \left(A\cup \hh{A}\right)\cup \hh{A} \\ & = & A\cup \hh{A} \\ & = & cl^*(A).\qedhere \end{array}$$ 
    
   %By  (1) of Theorem ~\ref{th5} and (2), we have     $cl^{*}(A)\subseteq cl^{*}(cl^{*}(A))$. So it suffices to show that for all $A\subseteq X$, we have $cl^{*}(cl^{*}(A))\subseteq cl^{*}(A)$ or equivalently that $x\notin cl^{*}(A) $, then $x\notin cl^{*}(cl^{*}(A)) $.
   %Now let $x\notin cl^{*}(A) $. Hence $x\notin A$ and $x\notin \hh{A}$. This  $x\notin A$ and $x\centernot\h{A}$, then by Theorem ~\ref{th1} implies that  $x\centernot\h \hh{A}$ and by (2) of Definition ~\ref{d1} we have $x\centernot\h (A\cup \hh{A})$, hence $x\centernot\h cl^{*}(A)$. Then by Lemma ~\ref{2} $x\notin \hh{cl^{*}(A)}$. Also $x\notin cl^{*}(A)\cup \hh{cl^{*}(A)}= cl^{*}(cl^{*}(A)) $. Hence $  cl^{*}(cl^{*}(A))=cl^{*}(A)$ which is completes the proof.
\end{proof}

%\begin{question}
%    In a primal-proximity space is the topology $\tau^{*}$ and $\hh{\tau}$ are independent? {\color{red}We will delete this question.}
%\end{question}

\begin{theorem}
Let $(X,\h, \mathcal{P})$ be a primal-proximity space. Then the following properties  hold:
\begin{enumerate}
           \item $B\centernot \h A$ if and only if $B\centernot \h cl^{*}(A)$.
           \item $cl^{*}\left(\hh{A}\right)=\hh{A}$.
           \item $cl^{*}\left(\hh{A}\right)=\hh{cl^{*}(A)}$.         \end{enumerate}
\end{theorem}
\begin{proof}
(1) Let $B\centernot \h A.$ Then, by Theorem \ref{th1}, we have $B\centernot \h \hh{A}$. Hence, by (2) of Definition \ref{d1},
$B\centernot \h (A\cup\hh{A})=cl^{*}(A)$ if and only if $B\centernot \h A$ and $B\centernot \h \hh{A}$.

(2)  Let $A\subseteq X.$ By (4) of Theorem \ref{th5}, we have $$cl^{*}\left(\hh{A}\right)=\hh{A}\cup\hh{\hh{A}}=\hh{A}.$$
%It is clear that $\hh{A}\subseteq cl^{*}(\hh{A})$.  Now, let $x\in cl^{*}(\hh{A})$. Then, $x\in \hh{A}$ or $\{x\}\h \hh{A}$. It follows that $x\in \hh{\hh{A}}$. Thus by lemma ~\ref{lem1}, we get $x\in \hh{A}$. Hence $cl^{*}(\hh{A})=\hh{A}$.

(3)  Let $A\subseteq X$. By (3) of Theorem \ref{th5}, we have $$cl^{*}\left(\hh{A}\right)=\hh{A}\cup \hh{\hh{A}}=\hh{\left(A\cup\hh{A}\right)}=\hh{cl^{*}(A)}.\qedhere$$
\end{proof}

\begin{theorem}
Let $(X,\h, \mathcal{P})$ be a primal-proximity space and $A,B,H\subseteq X$  such that $A\subseteq B.$  If $A\h B$ and $\{b\} \h H$ for all $b\in B$, then $A\h H.$
\end{theorem}
\begin{proof} Suppose $A\centernot \h H$, then  there exist $C, D\subseteq X$ such that $A \centernot \h C^{c}$, $D^{c}\centernot \h B$
and $(C\cap D)^{c}\notin \p$. This result, combined with $A\h B$ and (2) of Definition ~\ref{d1}, implies that $B\nsubseteq C^{c}$, that is $B\cap C\neq\emptyset$.
It follows that there is  a point  $x\in X$ such that $\{x\}\h H$ and $x\in C$. Then, there  are  two cases either $x\in D$ or $x\notin D$.

Case 1: If $x\in D$. Hence $X\setminus\{x\}\notin \p$, by (3) of Definition \ref{d1}, implies $\{x\}\centernot \h H$ for any subset $H$ of $X$, which
is contradiction.

Case 2: $x\in D^{c}$, then $\{x\}\color{black}\centernot \h B$. This result, combined with  (3) and (4) of Definition \ref{d1}, implies $\{x\} \centernot \h H$, which
is contradiction. Hence, $A\h H$.
\end{proof}
\begin{example}~\label{ex1}
Let $(X,\tau,\mathcal{P})$ be a primal topological space and $\h$ be a binary relation on $2^{X}$
defined as $A \h B$ if and only if $(A\cap cl(B))^{c}\in \p$. Then $``\h"$ is not a primal-proximity relation on $2^{X}$ but satisfes  (2),(3),(4) and (5) of Definition ~\ref{d1}. Hence, in this case $\tau\subseteq \tau^{*}$.
\end{example}

\begin{proof}
 We  want to show that $cl^{*}(A)\subseteq cl(A)$ for all $A\subseteq X$. Let $x\in cl^{*}(A)=A\cup \hh{A}$. Then, $x\in A$ or $x\in \hh{A}$. If $x\in A$, then $x\in cl(A)$. Now if $x\in \hh{A}$, then $\{x\}\h A$. Hence, $(\{x\}\cap cl(A))^{c}\in \p$ and so $(\{x\}\cap cl(A))^{c}\neq X$. Thus,  $\{x\}\cap cl(A)\neq \emptyset$ which means $x\in cl(A)$. Therefore, $\tau\subseteq \tau^{*}$.\color{black}
\end{proof}

\begin{example}~\label{ex2}
Let $(X,\tau,\mathcal{P})$ be a primal topological space and $\h$ be a binary relation on $2^{X}$
defined as $A \h B$ if and only if $(A\cap cl^{\diamond}(B))^{c}\in \p$. Then $``\h"$ is not a primal-proximity relation on $2^{X}$ but satisfy  (2), (3),  (4) and (5) of Definition ~\ref{d1}. Hence, in this case $\tau^{\diamond}\subseteq \tau^{*}$.
\end{example}

\begin{proof}
 We  want to show that $cl^{*}(A)\subseteq cl^{\diamond}(A)$ for all $A\subseteq X$. Let $x\in cl^{*}(A)=A\cup \hh{A}$. Then $x\in A$ or $x\in \hh{A}$. If $x\in A$, then $x\in A\subseteq A\cup A^{\diamond}=cl^{\diamond}(A)$. Now if $x\in \hh{A}$, then $\{x\}\h A$. Hence, $(\{x\}\cap cl^{\diamond}(A))^{c}\in \p$ and so $(\{x\}\cap cl^{\diamond}(A))^{c}\neq X$. Thus, $\{x\}\cap cl^{\diamond}(A)\neq \emptyset$ which means $x\in cl^{\diamond}(A)$. Therefore, $\tau^{\diamond}\subseteq \tau^{*}$.
\end{proof}

\begin{definition}
Let $(X,\tau,\mathcal{P})$ be a primal topological space. Then, $X$ is said to be a primal-regular space if for all $x\in X$ and  $\tau^{\diamond}$-closed set  $F$
such that $(\{x\}\cap F)^{c}\notin \p$ there exist two open sets $H,G$ such that $x\in  H$ and $F\subseteq G$ and
$(H\cap G)^{c}\notin \p$.
\end{definition}

\begin{theorem}
Let $(X,\tau,\mathcal{P})$ be a primal topological space. 
%%NO NEED THIS CONDITION \color{red} such that $\mathcal{P}=2^X\setminus \{X\}$. \color{black}
Let $X$ be a primal-regular space and $\h$ be a binary relation on $2^{X}$ as defined in  Example \ref{ex2}, then $\tau^{\diamond}=\tau^{*}$.

\end{theorem}
\begin{proof}
In order to prove the theorem, it suffices to show
$cl^{\diamond}(A)=cl^{*}(A)$ for all subsets $A$ of $X$.

Let $x\in cl^{*}(A)$. Then, $x\in A$ or $x\in \hh{A}$. If $x\in A$, then $x\in cl^{\diamond}(A)$. Now if $x\in \hh{A}$, then $\{x\}\h A$. Hence,
$(\{x\}\cap cl^{\diamond}(A))^{c}\in \p$ which means  $\{x\}\cap cl^{\diamond}(A)\neq\emptyset$. Consequently, we have $x\in cl^{\diamond}(A)$. Thus, $cl^{*}(A)\subseteq cl^{\diamond}(A)$.

Now, let $x\notin cl^{*}(A)$. Then, $x\notin A$ and $x\notin \hh{A}$. It follows that $\{x\} \centernot \h A$ and hence by Example \ref{ex2} implies  that
$(\{x\}\cap cl^{\diamond}(A))^{c}\notin \p$. Since $X$ is primal-regular space  and  $\tau^{c}\subseteq {\tau^{\diamond}}^{c}$,
there exist two open sets $H$ and $G$ such that $x\in H$ and $A\subseteq cl^{\diamond}(A)\subseteq G$ and
$(H\cap G)^{c}\notin \p$. Hence, $(H\cap A)^{c}\notin \p$ and since $x\in H\in \tau$ and $(H\cap A)^{c}\notin \p$, then $x\notin A^{\diamond}$.
So, $x\notin cl^{\diamond}(A)$. It follows that $cl^{\diamond}(A)\subseteq cl^{*}(A)$. Hence, $cl^{\diamond}(A)=cl^{*}(A)$.
\end{proof}

\begin{example} \label{ex3}
Let $(X,\tau,\mathcal{P})$ be a primal topological space and $\h$ be a binary relation on $2^{X}$
defined as $A \h B$ if and only if $(cl^{\diamond}(A)\cap cl^{\diamond}(B))^{c}\in \p$. Then, $``\h"$ is not a primal-proximity relation on $2^{X}$ but satisfies  (1)-(4)  of Definition \ref{d1}.
\end{example}

\begin{definition}
Let $(X,\tau,\mathcal{P})$ be a primal topological space. Then, $X$ is said to be a primal-normal space if for two $\tau^{\diamond}$-closed sets  $F_{1}, F_{2}$
such that $(F_{1}\cap F_{2})^{c}\notin \p$, there exist two open sets $H$ and $G$ such that $F_{1}\subseteq H$ and $F_{2}\subseteq G$ and
$(H\cap G)^{c}\notin \p$.
\end{definition}

\begin{theorem}
Let $(X,\tau,\mathcal{P})$ be a primal topological space. If $X$ is  a primal-normal space and   a binary relation defined as in  Example
\ref{ex3} and $(X,\tau)$ is $T_{1}$-space, then $\tau^{\diamond}=\tau^{*}$.
\end{theorem}

\begin{proof}
 In order to prove the theorem, it suffices to show that
$cl^{\diamond}(A)=cl^{*}(A)$ for all subsets $A$ of $X$.

Let $x\in cl^{*}(A)$. Then, $x\in A$ or $x\in \hh{A}$. If $x\in A$, then $x\in cl^{\diamond}(A)$. Now, if $x\in \hh{A}$, then $\{x\}\h A$. Hence,
$(cl^{\diamond}(\{x\})\cap cl^{\diamond}(A))^{c}\in \p$. Since $(X,\tau)$ is $T_{1}$-space and $\tau^{c}\subseteq {\tau^{\diamond}}^{c}$, then
$[\{x\}\cap cl^{\diamond}(A)]^{c}\in \p$ and so $\{x\}\cap cl^{\diamond}(A)\neq\emptyset$. Consequently, we have $x\in cl^{\diamond}(A)$. Hence, $cl^{*}(A)\subseteq cl^{\diamond}(A)$.

Now, let $x\notin cl^{*}(A)$. Then, $x\notin A$ and $x\notin \hh{A}$. It follows that $\{x\} \centernot \h A$ and hence by Example \ref{ex3} implies  that
$(cl^{\diamond}(\{x\})\cap cl^{\diamond}(A))^{c}\notin \p$. Since $(X,\tau)$ is primal-normal space, $T_{1}$-space and  $\tau^{c}\subseteq {\tau^{\diamond}}^{c}$,
there exist two open sets $H$ and $G$ such that $\{x\}\subseteq  H$, $A\subseteq cl^{\diamond}(A)\subseteq G$ and
$(H\cap G)^{c}\notin \p$. Hence, $(H\cap A)^{c}\notin \p$. Since $x\in H\in \tau$ and $(H\cap A)^{c}\notin \p$, thus $x\notin A^{\diamond}$.
So, $x\notin cl^{\diamond}(A)$. It follows that $cl^{\diamond}(A)\subseteq cl^{*}(A)$ and hence, $cl^{\diamond}(A)=cl^{*}(A)$.
\end{proof}

\begin{theorem}
Let $(X,\h, \mathcal{P})$ be a primal-proximity space and $A \subseteq X$. Then, $A\in \hh{\tau}$ if and only if $\{x\}\centernot \h A^{c}$ for every $x\in A$.
\end{theorem}

\begin{proof}
 Let $A\in \hh{\tau}$ and $x\in A$. Then, $A^{c}$ is proximity-closed in $\hh{\tau}$ and  $x\notin A^{c}$. Hence, we get $\{x\} \centernot\h A^{c}$.
%If $A\in \hh{\tau}$ and $x\in A$, then $A^{c}$ is proximity-closed in $\hh{\tau}$ and  $x\notin A^{c}$ implies $\{x\} \centernot\h A^{c}$, which shows that if $x\in A$, then $\{x\}\centernot \h A^{c}$.

Conversely, if for every  $x\in A$, we have $\{x\}\centernot \h A^{c}$, then  $\{x\} \h A^{c}$ implies that $x\notin A$.
This means that $\{x\}\h A^{c}$ implies $x \in A^{c}$. Hence, $A^{c}$ is proximity-closed in $\hh{\tau}$. Thus, $A\in \hh{\tau}$.
\end{proof}

\begin{theorem}\label{th2}
Let $(X,\h, \mathcal{P})$ be a primal-proximity space and $A,B\subseteq X$ such that $A \centernot\h B$. Then the following conditions hold:
\begin{enumerate}
  \item $cl_{\hh\tau}(B)\subseteq A^{c},$ where $cl_{\hh\tau}(B)$ means the closure of $B$ with respect to $\hh{\tau}$.
  \item  if $\mathcal{P}=2^X\setminus \{X\},$ \color{black} then $B\subseteq int_{\hh\tau}(A^{c})$ where $int_{\hh\tau}(A^c)$ means the interior of $A^c$ with respect to $\hh{\tau}$.
\end{enumerate}
%Here the closure and interior are taken with respect to $\hh{\tau}$.
\end{theorem}
\begin{proof}
(1) Since the closure is taken with respect to $\hh{\tau}$ and $A \centernot\h B$, we have $\hh{B}=cl_{\hh\tau}(B)\subseteq A^c$.

(2) If $x\in B$, then $\{x\}\h B$. This implies that $\{x\}\centernot \h A$. Because if $\{x\}\h A$, then $A\h B$ by Lemma \ref{lem2}. Hence, $x\notin cl_{\hh{\tau}}(A)$ which means $x\in (cl_{\hh{\tau}}(A))^{c}=int_{\hh\tau}(A^{c})$. Hence, we have $B\subseteq int_{\hh\tau}(A^{c})$. %\color{red} I will revise this proof by using hypothesis.
\end{proof}

\begin{theorem}
Let $(X,\h, \mathcal{P})$ be a primal-proximity space and $A, B \subseteq X$. Then,
 $A \h B$ if and only if  $cl_{\hh{\tau}} (A) \h cl_{\hh{\tau}} (B),$
where $cl_{\hh{\tau}} (A)$ means the closure of $A$ with respect to $\hh{\tau}$.
\end{theorem}
\begin{proof}
If $A\h B$, then by Lemma \ref{1}, $cl_{\hh{\tau}} (A)\h cl_{\hh{\tau}} (B)$ since $A\subseteq cl_{\hh{\tau}} (A)$ and $B\subseteq cl_{\hh{\tau}} (B)$.

If $A\centernot\h B$, then  there exists a subset $E$ of $X$ such that $A\centernot\h E$ and $E^{c}\centernot\h B$ and $(E\cap E^{c})^{c}\notin \p$.
Hence, $cl_{\hh{\tau}} (B)\subseteq E$ by (1) of Theorem \ref{th2}. This implies that $A\centernot\h cl_{\hh{\tau}}(B).$ Because if $A\h cl_{\hh{\tau}} (B)$ then by Lemma \ref{1}, then
$A\h E$ since $cl_{\hh{\tau}} (B)\subseteq E$. Now, if $A\centernot\h B$ then  $A\centernot\h cl_{\hh{\tau}}(B)$. Also, $cl_{\hh{\tau}}(B) \centernot\h A$ by similar  prove  again it follows that $cl_{\hh{\tau}} (B) \centernot\h cl_{\hh{\tau}} (A)$.
Hence, $A \h B$ if and only if  $cl_{\hh{\tau}} (A) \h cl_{\hh{\tau}} (B)$.

\section{Conclusion}

In  this paper, we introduced a new type of proximity space  called primal-proximity space. Later, we defined point-primal proximity operator and investigated some of its fundamental properties. We also proved that this operator is a Kuratowski closure operator under special condition. Moreover, one more operator via point-primal proximity operator was defined.  Furthermore, we gave not only some relationships but also several examples.

%% Above proof is OK.
%If $A\centernot\h B$, then  there exists a subset $E$ of $X$ such that $A\centernot\h E$ and $E^{c}\centernot\h B$ and $(E\cap E^{c})^{c}\notin \p$. Hence, $cl_{\hh{\tau}} (B)\subseteq E$ by (1) of Theorem \ref{th2}. This implies that $A\centernot\h \overline{B}$ for if $A\h cl_{\hh{\tau}} (B)$ then by Lemma \ref{1}, then $A\h E$ since $cl_{\hh{\tau}} (B)\subseteq E$. By applying Lemma \ref{1} again it follows that $cl_{\hh{\tau}} (A) \centernot\h cl_{\hh{\tau}} (B)$. Hence, $A \h B$ if and only if  $cl_{\hh{\tau}} (A) \h cl_{\hh{\tau}} (B)$. \color{red} According to me, there is a problem in the proof. I will think about it again.
\end{proof}

%%%%%%%%%%%%%%%%%%%%%%

%%%%%%%%%%%%%%%%%%%%%

%%%%%%%%%%%%%%%%%%%%%%%%%%%%%%%%%%%%%%%%%%%%%%%%%%%%%%%%%%%%%%%%%%%%%%%%%%%%%%%%%%%%%%%%%%%%%%
%%%%%%%%%%%%%%%%%%%%%%%%%%%%%%%%%
%%%%%%%%%%%%%%%%%%%%%%%%
%%%%%%%%%%%%%%a%%%%%%%%%%%%%%%%%%%%%%%%%%%%%%%%%%%%%%%%%%%%%%%%%%%%%%%%%%

{\bf Conflict of interest:} The authors declare that there is no conflict of interest.

\end{document}